\documentclass[a4paper, 12pt]{article}

\usepackage[sort&compress]{natbib}
\bibpunct{(}{)}{;}{a}{}{,} 

\usepackage{amsthm, amsmath, amssymb, mathrsfs, multirow, url, subfigure}
\usepackage{graphicx} 
\usepackage{ifthen} 
\usepackage{amsfonts}
\usepackage[usenames]{color}
\usepackage{fullpage}
\usepackage{tikz}

\theoremstyle{plain} 
\newtheorem{thm}{Theorem}

\newtheorem{prop}{Proposition}

\theoremstyle{definition}
\newtheorem{defn}{Definition}
\newtheorem*{asmp}{Model Assumptions}

\theoremstyle{remark}

\newtheorem*{dpbayes}{Default-prior Bayes}
\newtheorem*{gfid}{Generalized fiducial}

\newcommand{\prob}{\mathsf{P}}

\newcommand{\unif}{{\sf Unif}}

\newcommand{\RR}{\mathbb{R}}

\newcommand{\XX}{\mathbb{X}}
\newcommand{\YY}{\mathbb{Y}}

\newcommand{\UU}{\mathbb{U}}

\newcommand{\TT}{\mathbb{T}}

\newcommand{\G}{\mathscr{G}}
\newcommand{\Gbar}{\overline{\mathscr{G}}}
\newcommand{\K}{\mathscr{K}}

\newcommand{\eps}{\varepsilon}

\newcommand{\prior}{\mathsf{Q}}

\newcommand{\cred}{\mathscr{C}}

\newcommand{\lPi}{\underline{\Pi}}
\newcommand{\uPi}{\overline{\Pi}}

\title{Fiducial inference viewed through a possibility-theoretic inferential model lens}
\author{Ryan Martin\footnote{Department of Statistics, North Carolina State University, {\tt rgmarti3@ncsu.edu}}}
\date{\today}

\begin{document}

\maketitle 

\begin{abstract}
Fisher's fiducial argument is widely viewed as a failed version of Neyman's theory of confidence limits.  But Fisher's goal---Bayesian-like probabilistic uncertainty quantification without priors---was more ambitious than Neyman's, and it's not out of reach.  I've recently shown that reliable, prior-free probabilistic uncertainty quantification must be grounded in the theory of imprecise probability, and I've put forward a possibility-theoretic solution that achieves it.  This has been met with resistance, however, in part due to statisticians' singular focus on confidence limits.  Indeed, if imprecision isn't needed to perform confidence-limit-related tasks, then what's the point?  In this paper, for a class of practically useful models, I explain specifically why the fiducial argument gives valid confidence limits, i.e., it's the ``best probabilistic approximation'' of the possibilistic solution I recently advanced.  This sheds new light on what the fiducial argument is doing and on what's lost in terms of reliability when imprecision is ignored and the fiducial argument is pushed for more than just confidence limits.  

\smallskip

\end{abstract}

\section{Introduction}
\label{S:intro}

In the early 20th century, statistical inference and Bayesian inverse probability were synonymous.  Fisher, dissatisfied with the Bayesians' insistence on an {\em a priori} distribution and their default use of flat prior distributions, put forward a novel alternative---the {\em fiducial argument} \citep[e.g.,][]{fisher1930, fisher1933, fisher1935b, fisher1935a}.  At a high level, the fiducial argument takes the model-based probabilities for the observable data, depending on fixed parameters, and {\em flips} them into probability statements about the unknown parameters, depending on the observed data; see, e.g., \citet{zabell1992}, \citet{seidenfeld1992}, and \citet{dawid2020} for details.  Fisher's proposal was too good to be true in general, so, naturally, both supporters and skeptics carefully scrutinized Fisher's claims, some even attempted reformulating Fisher's proposal.  Most notable is Neyman's theory of {\em confidence limits}, which Neyman himself considered to be ``an extension of the previous results of Fisher'' \citep{neyman1941}.  Besides Neyman, the work of Dempster \citep[e.g.,][]{dempster1966, dempster1967, dempster1968a} aimed at fixing/extending the fiducial argument, and was a key development in imprecise probability theory.  More recently, Hannig and others have developed a theory of {\em generalized fiducial inference} in which the ``fiducial flip'' can be applied more systematically; see, e.g., \citet{hannig.review}, \citet{murph.etal.fiducial}, and Section~\ref{SS:fid}.

One thing that Fisher and Neyman would've agreed on is that the fiducial argument and the theory of confidence intervals are distinct.  Still, much of the debate surrounding the two theories focused on whether the fiducial argument produced confidence limits.  Fisher said yes, he was later proved wrong, hence ``Fisher's biggest blunder'' \citep{efron1998}.  It's surprising to me that Fisher didn't shift the conversation away from confidence limits.  Even if the fiducial-limits-are-confidence-limits claim was correct, that was a losing battle: a particular confidence limit construction, even an ingenious one, can't do better than the ``best'' confidence limits available in a given application.  

Lost in this singular confidence-limit focus is that Fisher's goal was more ambitious: it aimed to provide Bayesian-like probabilistic uncertainty quantification in the absence of prior information.  That goal remains the ``most important unresolved problem in statistical inference'' \citep{efron.cd.discuss}, but, with a few exceptions \citep[e.g.,][]{taraldsen.lindqvist.2013}, modern efforts in this direction remain focused solely on confidence-limit related questions.  This is still a losing battle, I think, so a different perspective is needed.  

Simply put, a full-blown probability distribution isn't needed to get confidence limits.  Instead, the construction of data-dependent probability distributions for inference is motivated by a desire to quantify uncertainty more broadly.  In that case, our evaluation of proposed solutions ought to be consistent with these broader objectives, so the metrics I've been advocating for take the (data-dependent) probability as the primitive, and require that the ``probabilities'' this distribution assigns to true/false hypotheses tend to not be small/large.  The rationale is that judgments will be made based on the magnitudes of these probabilities, so the above condition implies that these inferences would be reliable or {\em valid} in a certain sense; see Section~\ref{SS:im}.  It's only after realizing that precise probability distributions are incapable of achieving this kind of reliability that the need for imprecise-probabilistic or {\em possibilistic} considerations.  That's the motivation behind the possibility-theoretic {\em inferential model} (IM) framework that I've been advocating for recently \citep[e.g.,][]{imposs, imchar, martin.partial2}.

Despite the clear differences between fiducial and IMs, there are some superficial similarities.  In particular, it's common for the confidence limits derived from an IM to match those derived from a fiducial-like solution.  Some authors \citep[e.g.,][]{prsa.conf, cui.hannig.im} have recognized this connection and drawn the conclusion that the IM's imprecision is unnecessary, that it's somehow enough to work with fiducial/confidence distributions and, more generally, precise probability theory.  These arguments are squarely focused on confidence limits, so they overlook the fact that the ``confidence'' property satisfied by ``confidence distributions'' isn't preserved under the probability calculus.  For a theory of ``confidence distributions'' to do anything more than produce confidence limits, i.e., to not be a losing battle in the sense above, it {\em must} have its foundations in the theory of imprecise probability.  

I've already responded to some of these critiques \citep{imchar, prsa.response}, but I have some new insights to share that will help clear up this confusion about confidence limits versus (imprecise) probabilistic inference.  Specifically, my goal here is to give a complete characterization of the relationship between the fiducial solution---which agrees with the default-prior Bayes solution, among others, in the context I'm considering---and the IM solution.  That is, following some background in Section~\ref{S:background} and a description of the models under consideration (Section~\ref{SS:invariant}), I show in Section~\ref{SS:main} that the fiducial solution is the maximal, inner probabilistic approximation to the IM's possibility measure output.  This connection, along with the IM's validity property, explains why the fiducial argument returns genuine confidence limits---agree with the IM's limits---in this class of problems.  It also sheds light on why the IM's imprecision is needed, i.e., the maximal inner probabilistic approximation of the marginal IM need not be the corresponding marginal fiducial distribution.  It's for this and other similar reasons (Section~\ref{S:new}) that no fiducial-like argument grounded in precise probability can resolve Efron's ``most important unresolved problem.'' Finally, in Section~\ref{S:false.confidence}, I offer a partial answer to a key question: which hypotheses are not afflicted by {\em false confidence} \citep{balch.martin.ferson.2017, martin.nonadditive} relative to the fiducial distribution?

\section{Background}
\label{S:background}

\subsection{Problem setup}
\label{SS:setup}

Let $X \in \XX$ denote observable data---could be a scalar, a vector, a matrix, a collection scalars, vectors, or matrices, or something else entirely.  The standard textbook case of a sample from some population is covered by this general setup, as are many others.  Let $\prob_\theta$ denote a posited statistical model for the observable $X$, depending on a parameter $\theta \in \TT$.  Again, the parameter $\theta$ is very general, but it's typically a scalar or a vector.  Of course, the parameter is {\em unknown} and, as is customary, when I'm referring to the uncertain variable I'll write it as $\Theta$, saving the symbol $\theta$ for its particular values.  

The goal is to make inference about $\Theta$ based on {\em only} the observed value $x$ of $X$ and the structure included in the posited model.  By ``make inference'' I mean quantify uncertainty about $\Theta$, given $X=x$, via a (precise or imprecise) probability distribution supported on $\TT$.  Of course, there are many ways that this can be carried out, and I'll describe those most relevant to this paper below.  First I have to stress that there is {\em no prior knowledge about $\Theta$ assumed here}, i.e., the prior about $\Theta$ is vacuous.  While I don't believe a vacuous-prior-knowledge assumption is realistic in most applications \citep{martin.partial}, this is the typical starting point in the literature.  This means that neither ordinary Bayesian inference with a single prior distribution \citep{berger1985, bernardo.smith.book, ghosh-etal-book} nor generalized Bayesian  inference with a proper subset of all prior distributions \citep{walley1991, augustin.etal.bookchapter} are viable options to achieve the desired goal.  Other solutions, like the ones I describe briefly below, are needed.

\subsection{Fiducial-like constructions}
\label{SS:fid}

Fisher's original fiducial argument can be directly applied to only a relatively narrow range of problems.  It's possible to transform other problems into ones to which the fiducial argument applies, but this becomes more challenging when the parameter $\theta$ is multivariate.  There's a general---but still relatively narrow---class of problems for which a version of Fisher's fiducial argument applies; see Section~\ref{S:connection} for details.  

Despite Fisher's failure to formulate a fully general ``theory of fiducial inference,'' the idea is so appealing that many others have advanced their own versions.  Of course, some of these alternatives are more consistent with Fisher's vision than others.  I'll give details here about only two of these alternatives.
\begin{dpbayes}
The basic idea is to take the posterior distribution to have a density (with respect to Lebesgue measure on $\TT$) given by 
\[ q_x(\theta) \propto L_x(\theta) \, q(\theta), \quad \theta \in \TT, \]
where $L_x$ is the likelihood function based on data $X=x$, $q$ is some non-negative function, not necessarily a probability density, and the proportionality constant is determined by integrating the expression in the above display.  Of course, if that integral diverges, which is possible when $q$ isn't a density, then the posterior isn't well defined.  This idea goes back at least to Laplace, who suggested that, in the absence of genuine subjective prior input for $\Theta$, one can take a flat, uniform prior distribution on $\TT$.  The uniform prior's lack of reparametrization invariance was later resolved by \citet{jeffreys1946}. While Jeffreys's class of priors is widely used, it has two shortcomings: first, the Fisher information matrix it depends on doesn't always exist \citep[e.g.,][]{shemyakin2014, lin.martin.yang.hellinger} and, second, a ``good'' default prior for $\Theta$ can't support ``good'' posteriors for all features of $\Theta$ \citep[e.g.,][]{fraser2011, fraser.etal.2016}.  {\em Reference priors} \citep[e.g.,][]{bergerbernardosun2009, bernardo1979, berger2006} aim to overcome these issues.  Finally, since these priors depend on the posited statistical model, Bayesian inference based on them violates the likelihood principle \citep[e.g.,][]{basu1975, birnbaum1962, bergerwolpert1984}. 
\end{dpbayes}

\begin{gfid}
Following \citet{hannig.review} and the references therein, the starting point is the so-called data-generating equation
\begin{equation}
\label{eq:dge}
X = a(\Theta, Z), \quad Z \sim \prob,
\end{equation}
where $a$ is a known function and $\prob$ is a known probability distribution.  This is familiar in the context of data simulation.  In the inference setting, this expression effectively links the observable data $X$ to the uncertain $\Theta$ through a random variable with known distribution.  This is the basic starting point taken by many different frameworks, including those in \citet{fraser1968}, \citet{dawidstone1982}, and \citet{dempster2008}.  Hannig defines the {\em generalized fiducial distribution} for $\Theta$, given $X=x$, as the weak limit (as $\eps \to 0$) of the random variable that solves the constrained optimization problem
\[ \arg\min_{\theta \in \TT} d\bigl( X, a(\theta, Z) \bigr) \quad \text{subject to} \quad \min_{\theta \in \TT} d\bigl( X, a(\theta,Z) \bigr) \leq \eps, \]
where $d$ is an appropriate distance measure.  I'll focus here on the case where the above solution is unique, but that's not necessary for the theory.  When other suitable conditions are met, there's a simple formula for the generalized fiducial density, i.e., 
\begin{equation}
\label{eq:gf.density}
q_x(\theta) \propto L_x(\theta) \, J(x,\theta),
\end{equation}
where $J(x,\theta)$ is basically a Jacobian term resulting from the transformation of $(Z,\prob)$ to the solution $\Theta$ of the above optimization problem.  Despite the obvious similarities, there's a key difference between Bayes's rule and the right-hand side of \eqref{eq:gf.density}: $J$ is not a prior density---it might even depend on data---that's chosen by the user, it's completely determined by the posited model (and the distance $d$).  
\end{gfid}

The two very different approaches described above produce inferential output that looks very similar.  In fact, for the class of problems considered in Section~\ref{S:connection}, their respective solutions are the same.  More generally, both return probability distribution with density functions determined by the likelihood times a ``weight function,'' which might depend on data.  Given their similar forms, one might expect the two solutions to have similar properties.  As is common, let's consider large-sample case where $x^n=(x_1,\ldots,x_n)$ is an iid sample of size $n \to \infty$ from a common distribution depending on a parameter $\Theta \in \TT \subseteq \RR^D$.  When standard regularity conditions are satisfied, both the above solutions have a corresponding {\em Bernstein--von Mises theorem}, which goes as follows.  Let ${\sf G}_{x^n}$ denote the $D$-dimensional Gaussian distribution with mean equal to the maximum likelihood estimator and covariance matrix equal to $n^{-1}$ times the (observed) Fisher information matrix.  Then the theorem states that the total variation distance between $\prior_{X^n}$ and ${\sf G}_{X^n}$ is vanishing (in $\prob_\Theta$-probability) as $n \to \infty$; for a proof in the Bayesian case, see \citet[][Ch.~10.2]{vaart1998}, and details for the generalized fiducial distribution can be found in \citet{hannig.review} and the references therein.  

The importance of the total-variation distance is that it's a strong enough metric to ensure that confidence limits derived from the Bayes or fiducial distributions are, in fact, confidence limits, at least approximately as $n \to \infty$.  For the most part, the theory behind these solutions stops here---that's because the focus is squarely on confidence limits.  As I explained above, a focus on confidence limits is a losing battle in the sense that it'll fall short of resolving the ``most important unresolved problem.'' To get over this hump, a broader perspective on uncertainty quantification is needed.

\subsection{Inferential models}
\label{SS:im}

The {\em inferential model} (IM) formulation, first developed in \citet{imbasics, imbook}, has output that takes the mathematical form of a possibility distribution or, equivalently, a consonant belief function.  Their motivation for the introduction of imprecision was that, in order to be reliable when quantifying uncertainty more broadly than with confidence limits, one needs to be more conservative.  The {\em possibilistic} brand of imprecision, compared to other imprecise probability models (e.g., lower previsions), is ideally suited to achieve the error rate control properties that statisticians desire.  

More recently, \citet{martin.partial} developed a simpler and more flexible IM construction, one that makes direct use of the posited model's likelihood function rather than a data-generating equation like in \eqref{eq:dge}.  This new formulation allows for the incorporation of available partial prior information about $\Theta$, but here I'm assuming no prior information.  Let $L_x(\theta)$ denote the likelihood function, and define the {\em relative likelihood}
\begin{equation}
\label{eq:eta}
R(x,\theta) = \frac{L_x(\theta)}{\sup_{\vartheta \in \TT} L_x(\vartheta)}, \quad \theta \in \TT. 
\end{equation}
Next, define the {\em possibility contour} 
\begin{equation}
\label{eq:contour}
\pi_x(\theta) = \prob_\theta\{ R(X,\theta) \leq R(x,\theta)\}, \quad \theta \in \TT. 
\end{equation}
Note that $\pi_x$ has maximum value 1, which is attained at a maximum likelihood estimator $\hat\theta_x \in \arg\max_\theta L_x(\theta)$.  The contour function determines the IM's imprecise-probabilistic output via the formulae 
\begin{align}
\uPi_x(A) & = \sup_{\theta \in A} \pi_x(\theta) \label{eq:upper} \\
\lPi_x(A) & = 1 - \uPi_x(A^c), \quad A \subseteq \TT. \notag
\end{align}
Here $\lPi_x$ and $\uPi_x$ are conjugate lower and upper probabilities or, more specifically, $\lPi_x$ and $\uPi_x$ are necessity and possibility measures, respectively.  
I'll call the IM's output {\em possibilistic} to emphasize the fact that $\uPi_x$ is a possibility measure.  Details concerning interpretation of the IM's output can be found in \citet{martin.partial, martin.partial2}; details on computation can be found in \citet{syring.martin.isipta21} and  \citet{hose.hanss.martin.belief2022}.  

The possibility contour expression in \eqref{eq:contour} is familiar---it's the p-value for a likelihood ratio test of the null hypothesis $H_0: \Theta=\theta$---but that's not the reason why the contour is defined in this way.  There's what I believe to be a principled construction that justifies this choice; see \citet[][Sec.~4]{martin.partial2} for details.  

The IM framework originated from the desire to achieve the best of both worlds: Bayesian-like ``probabilistic'' uncertainty quantification and inference with frequentist-style calibration and error rate control guarantees.  Indeed, of primary importance is that the IM's output be {\em valid} in the sense that the upper probability $x \mapsto \uPi_x$ satisfy 
\begin{equation}
\label{eq:valid}
\sup_{\theta \in A} \prob_{\theta}\{ \uPi_X(A) \leq \alpha \} \leq \alpha, \quad \alpha \in [0,1], \quad A \subseteq \TT. 
\end{equation}
There's an equivalent condition in terms of the lower probability, but it's not needed here.  The intuition behind \eqref{eq:valid} is as follows: assigning a relatively small upper probability to a true assertion might lead to an erroneous inference, so \eqref{eq:valid} helps protect the decision-maker by ensuring that such events are relatively rare.  This is the IM's reliability guarantee.  It's easy to see that the IM with upper probability $\uPi_x$ in \eqref{eq:upper} satisfies \eqref{eq:valid}: it follows immediately from the inequality $\uPi_X(A) \geq \pi_X(\Theta)$, for all $A$ that contain $\Theta$, and the fact that $\pi_X(\Theta)$ is stochastically no smaller than $\unif(0,1)$.  It's important to note that \eqref{eq:valid} is different from the usual frequentist Type~I error control in testing.  Validity covers all $A$'s at once, hence ensuring reliability even when interest is in a (potentially non-linear) feature $\psi(\Theta)$ of $\Theta$; just take, e.g., $A = \{\theta: \psi(\theta) \leq 7\}$. 

Finally, adjustments can be made to the IM construction above that do not affect the validity and, in fact, can often help to improve its efficiency.  One obvious adjustment is that it suffices to work with a minimal sufficient statistic.  More than that, for cases where the minimal sufficient statistic has an ancillary component---like in the class of problems considered below---I showed \citep[][Sec.~6]{martin.partial2} that the IM construction should be carried out conditional on the observed value of the ancillary statistic.  Further details on this will be given in Section~\ref{SS:im.group}.

\section{A fiducial--IM connection}
\label{S:connection}

\subsection{Invariant statistical models}
\label{SS:invariant}

Let $\G$ denote a group of bijections $g: \XX \to \XX$ acting on $\XX$, with function composition $\circ$ as the binary operation.\footnote{If $\XX$ is a product space, then $\G$ can be extended, if necessary, by applying it coordinate-wise.}  As is customary in the literature, I'll write $gx$ for the image of $x \in \XX$ under transformation $g \in \G$; and if $g_1$ and $g_2$ are two group elements, then $g_1 \circ g_2$ denotes their composition.  Since $\G$ is a group, it's associative, i.e., $g_1 \circ (g_2 \circ g_3) = (g_1 \circ g_2) \circ g_3$ for all $g_1,g_2,g_3 \in \G$, it contains the identity transformation $e$, and for every $g \in \G$, there exists an inverse $g^{-1} \in \G$ such that $g \circ g^{-1} = g^{-1} \circ g = e$. Some examples include location shifts, rescaling, rotations, affine transformations, and permutations.  Note that the group $\G$ is typically associated with a finite-dimensional space, e.g., rotations in Euclidean space correspond to unit-determinant matrices.  

The group $\G$ connects to the statistical model as follows.  Suppose that, for each $g \in \G$ and each $\theta \in \TT$, there exists a corresponding $\bar g \theta \in \TT$ such that 
\begin{equation}
\label{eq:invariant}
\prob_\theta(gX \in B) = \prob_{\bar g \theta}(X \in B), \quad \text{$(\theta,g) \in \TT \times \G$, measurable $B \subseteq \XX$}. 
\end{equation}
The most common example of this is where the distribution of $X$ depends on a location parameter $\theta$ and, consequently, the distribution of $X+a$ depends on the location parameter $\theta + a$.  When the statistical model $\{\prob_{\theta}: \theta \in \TT\}$ satisfies \eqref{eq:invariant}, it's called an {\em invariant statistical model}.  A good textbook presentation on this is \citet[][Ch.~6]{schervish1995}; a more comprehensive account of the theory and applications of invariant statistical models is \citet{eaton1989}.  Here I'll present only the necessary details. 


Define $\Gbar$ as the collection of all those bijections $\bar g: \TT \to \TT$, corresponding to the mappings $g \in \G$.  It is easy to check that $\Gbar$ is itself a group.   As is often the case in applications, I'll assume that the distributions $\prob_{\theta}$ all have a density respect to some underlying $\sigma$-finite measure.  So I'll follow the suggestion\footnote{Eaton's Theorem~3.1 says that \eqref{eq:invariant.density} implies \eqref{eq:invariant} but, the converse is only ``almost'' true. That is, \eqref{eq:invariant} and existence of densities implies \eqref{eq:invariant.density} on a null set that can depend on $(\theta,g)$. Eaton goes on to say that this null set dependence on $(\theta,g)$ can be avoided in all the applications that he's familiar with.} from Theorem~3.1 in \citet{eaton1989} and say that the family $\{p_\theta: \theta \in \TT\}$ is invariant with respect to $\G$ if 
\begin{equation}
\label{eq:invariant.density}
p_\theta(x) = p_{\bar g \theta}(gx) \, \chi(g), \quad x \in \XX, \; \theta \in \TT, \; g \in \G, 
\end{equation}
where $\chi(g)$ is the ``multiplier,'' a change-of-variables Jacobian term.  To be clear, these model assumptions aren't necessary to define the fiducial, Bayes, and IM solutions; this just sets a context in which a connection between these solutions can be made. 

The set $\G x = \{gx: g \in \G\} \subseteq \XX$ is called the {\em orbit} of $\G$ corresponding to $x$.  The orbits partition $\XX$ into equivalence classes, so every point $x \in \XX$ falls on exactly one orbit.  This partition can be used to construct a new coordinate system on $\XX$ which will be useful for us in what follows.  Identify $x \in \XX$ with $(g_x, u_x)$, where $u_x \in \UU$ denotes the label of orbit $\G x$ and $g_x \in \G$ denotes the position\footnote{The position is relative to some predetermined reference point on the orbit.  If $r_x$ is that reference point on $\G x$, then $g_x$ is defined such that $x = g_x r_x$.} of $x$ on the orbit $\G x$. 

Henceforth, I'm going to sacrifice a bit of generality for the sake of readability.  As is common in the literature, instead of just assuming that $\G$ and $\Gbar$ are isomorphic, I'm going to follow \citet[][p.~371]{schervish1995} and assume that $\TT = \G = \Gbar$; this means we don't have to distinguish $g$, $\bar g$, and $\theta$ and we don't have to track functions that connect the three.  With this context in mind, let's agree on the following 

\begin{asmp}
Let $\{p_\theta: \theta \in \TT\}$ be a family of densities invariant with respect to a locally compact topological group $\G$ in the sense of \eqref{eq:invariant.density} and, as explained above, take $\TT = \G = \Gbar$.  In addition, the following hold:
\begin{itemize}
\item[A1.] The left Haar measure $\lambda$ and the corresponding right Haar measure $\rho$ on (the Borel $\sigma$-algebra of) $\G$ exist and are unique up to scalar multiples.
\vspace{-2mm}
\item[A2.] There exists a bijection $t: \XX \to \G \times \UU$, with both $t$ and $t^{-1}$ measurable, that maps $x \in \XX$ to its position--orbit coordinates $(g_x, u_x) \in \G \times \UU$.  
\vspace{-2mm}
\item[A3.] The distribution of $t(X) = (G, U) \in \G \times \UU$ induced by the distribution of $X \sim \prob_{X|\theta}$ has a density with respect to $\lambda \times \mu$ for some measure $\mu$ on $\UU$.
\end{itemize} 
\end{asmp}

A few quick remarks about these assumptions are in order.  First, it's not necessary that $\TT = \G = \Gbar$, only that they're effectively the same, i.e., $\G$ and $\Gbar$ are isomorphic and that there's a certain bijection that relates $\TT$ and $\Gbar$. Next, for A1, existence and uniqueness of the left and associated right Haar measures on locally compact topological groups is a classical result and I'll refer the reader to the corresponding classical texts: \citet{halmos.measure} and \citet{nachbin1965}.  For A2, note that $t(gx) = (g \circ g_x, u_x)$ for all $g \in \G$.  That is, $g$ only acts on the first coordinate in $t$, so it's invariant with respect to $\G$ in the second coordinate---the orbit label is unaffected by transformations in $\G$.  Finally, for A3, existence of a joint density with respect to a product measure simply ensures that there will be no difficulty in defining a conditional distribution for $G$, given $U=u$; moreover, $U=U_X$, as a function of $X \sim \prob_{\theta}$ is an ancillary statistic.


Perhaps the simplest example is that of a location parameter where $\G = \Gbar = (\RR, +)$.  Since this group is abelian, the left and right Haar measures are the same and both equal to Lebesgue measure.  The function $x \mapsto t(x)$ in A2 consists of two components: in its ``$g_x$'' coordinate an equivariant function of $x$ that estimates the location and, in its ``$u_x$'' component, an invariant function of $x$, such as residuals.  For example, $g_x = \bar x$ the arithmetic mean of $x=(x_1,\ldots,x_n)$ and $u_x = \{x_i - \bar x: i=1,\ldots,n\}$.  Note that the $u_x$ coordinate satisfies a constraint, so, after it's represented in a suitable lower-dimensional space $\UU$, $\mu$ can be taken as Lebesgue measure there.  This example is insightful but also too simple.  There are many other problems that fit this general form; see, e.g., Section~\ref{SS:circle} below and Chapters~1--2 of \citet{fraser1968}, including the exercises.


\subsection{Fiducial and IM solutions}
\label{SS:im.group}

Under the above-described invariant statistical model setup, there is a standard/accepted fiducial distribution construction, which I'll describe below.  It also turns out that the IM solution as presented in Section~\ref{SS:im} is relatively straightforward in this case too.  

Just a quick remark on notation before getting started.  Recall that, for simplicity, we're assuming $\TT = \G = \Gbar$.  In this case, generic values of $\theta \in \TT$ can be identified as transformations in $\G$; the same can be said for the unknown value of the uncertain variable $\Theta$.  So, in what follows, I'll treat $\theta$ as a transformation that maps $x$ to $\theta x$, can be inverted to $\theta^{-1}$, and can be composed via $\circ$ with other transformations in $\G$.  

The key result here is that presented in Corollary~6.64 of \citet{schervish1995}.  In particular, it says that the density of $X$ under $\prob_{\theta}$ or, equivalently, the joint density of $t(X) = (G,U) \in \G \times \UU$, is given by 
\begin{equation}
\label{eq:GU.joint}
p_\theta(g, u) = f(\theta^{-1} \circ g, u), \quad \text{$(\lambda \times \mu)$-almost all $(g,u)$}, 
\end{equation}
where $f: \G \times \UU \to \RR$ is a fixed function that doesn't directly depend on $\theta$.  The particular form of $f$ isn't important for us at the moment---all that matters is how it depends on $\theta$.  

There are two important consequences of the characterization \eqref{eq:GU.joint}.  The first is presented as Lemma~6.65 in \citet{schervish1995}.

\begin{prop}
\label{prop:group.fid}
The fiducial distribution $\prior_x$ of $\Theta$, given $X=x$ or, equivalently, given $t(X)=(g,u)$, has a density $q_x$ with respect to right Haar measure on $\G$ given by 
\begin{equation}
\label{eq:fid.density}
q_x(\theta) = c_g \, p_\theta(g \mid u), \quad \theta \in \TT, \quad \text{with $t(x)=(g,u)$}, 
\end{equation}
where $c_g$ is a constant that depends only on $g$, and $p_\theta(g \mid u)$ is the conditional density of $G$, given $U=u$, derived from the joint density in \eqref{eq:GU.joint}.  Also, the fiducial density in \eqref{eq:fid.density} is the same as the Bayesian posterior density under the right Haar prior.  
\end{prop}

Here's what's going on behind the scenes.  If $G$ is as in $t(X)=(G,U)$, where $X \sim \prob_{\theta}$, then the conditional distribution of $H := \theta^{-1} \circ G$ in $\G$, given $U=u$, doesn't depend on $\theta$. Therefore, $H$ is a {\em pivot}, a function of $(X,\theta)$ having distribution free of $\theta$.  Then the fiducial distribution for $\Theta$, given $X=x$ or, equivalently, given $(G,U)=(g,u)$, is 
\[ \prior_x(A) = \prob(g \circ H^{-1} \in A \mid U=u), \quad A \subseteq \TT, \]
where the probability on the right-hand side is with respect to the  conditional distribution of $H$, given $U=u$, derived from $\prob_\theta$, which doesn't depend on $\theta$.  If that conditional distribution is known, then the fiducial probabilities can be directly evaluated or approximated via Monte Carlo. The Bayes connection can be of some practical benefit, instead of working with the $(H \mid U=u)$ conditional distribution, one can apply any of the Monte Carlo methods commonly used to approximate Bayesian posterior distributions.  

The next result, which is related to the ideas in the previous paragraph, is the second important consequence of \eqref{eq:GU.joint}, relevant to the IM solution. 

\begin{prop}
\label{prop:group.rellik}
The relative likelihood function $R(x,\theta)$ in \eqref{eq:eta} depends on $(\theta^{-1} \circ g, u)$ only, where $t(x)=(g,u)$.  In particular, if $t(x) = (g,u)$, then 
\[ R(x,\theta) = d_u \, f(\theta^{-1} \circ g, u), \]
where $d_u$ is a constant that depends only on $u$.  
\end{prop}

\begin{proof}
Since $x \mapsto t(x)$ is a bijection, by \eqref{eq:GU.joint} we have 
\[ R(x,\theta) = \frac{p_\theta(g,u)}{\sup_{\vartheta \in \TT} p_\vartheta(g,u)} = \frac{f(\theta^{-1} \circ g, u)}{\sup_{\vartheta \in \TT} f(\vartheta^{-1} \circ g, u)}. \]
Since $\vartheta$ is free to vary across all of $\TT$ or, equivalently, all of $\G$, so too is $\vartheta^{-1} \circ g$.  This implies that the supremum on the right-hand side doesn't depend on $g$.  The claim follows by taking $1/d_u$ equal to that supremum. 
\end{proof}

The key observation is that, since $R(X,\theta)$ depends on $X$ only through $(\theta^{-1} \circ G, U)$, and $H := \theta^{-1} \circ G$ is a pivot, it follows that $R(X,\theta)$ is a pivot too.  The relative likelihood is a pivot in many of the familiar textbook problems, so this is not a surprising result.  Even beyond these invariant statistical models, Wilks's theorem implies that the relative likelihood is an asymptotic pivot, under certain regularity conditions.  In any case, it's now straightforward to write down the IM solution's contour function:
\[ \pi_x(\theta) = \prob\{ f(H, u) \leq f(\theta^{-1} \circ g, u) \mid U=u\}, \quad \theta \in \TT, \]
where the probability on the right is the conditional distribution of $H$, given $U=u$, derived from $\prob_\theta$, which doesn't depend on $\theta$.  Analogous to the formulation in Section~\ref{SS:im}, I can let $\uPi_x(A) = \sup_{\theta \in A} \pi_x(\theta)$ define the IM's possibility measure output. Where it's helpful, below I'll write ``$\pi_{g|u}$'' and ``$\uPi_{g|u}$'' to emphasize that the IM output depends on the observed $x$ only through $t(x)=(g,u)$, and that $u$ is fixed at its observed value. With this notation, the contour clearly satisfies
\[ \prob_{\Theta}\{ \uPi_{G|u}(\Theta) \leq \alpha \mid U=u\} \leq \alpha, \quad \alpha \in [0,1], \]
where $(G,U)$ on the right-hand side depend implicitly on $X \sim \prob_\Theta$. From here we get the following conditional version of the validity property \eqref{eq:valid}: 
\[ \sup_{\theta \in A} \prob_\theta\{ \uPi_{G|u}(A) \leq \alpha \mid U=u \} \leq \alpha, \quad \alpha \in [0,1], \quad A \subseteq \TT. \]

\subsection{Connection}
\label{SS:main}

Roughly speaking, the main result here says that, in the invariant statistical model setting, the fiducial distribution $\prior_x$ described above is not only a member of the IM's credal set $\cred(\uPi_x)$ but is a ``maximal'' such member in the sense of being most diffuse, having the heaviest tails, etc.  First I need to define this notion of maximality precisely.  This will be defined in terms of a generic uncertain variable $Y$, taking values in a space $\YY$, about which uncertainty is quantified via a possibility measure $\Pi$.  Recall that the credal set $\cred(\Pi)$ for a possibility measure $\Pi$ supported on $\YY$ is the set of all probability measures it dominates, i.e., $\cred(\Pi) = \{\prob \in \text{prob}(\YY): \prob(\cdot) \leq \Pi(\cdot)\}$.

\begin{defn}
\label{def:maxdif}
Given a possibility measure $\Pi$ on a space $\YY$ with contour $\pi$, a probability measure $\prob$ supported on $\YY$ is a {\em maximal} element in the credal set $\cred(\Pi)$ if 
\[ \prob\{\pi(Y) \leq \alpha\} = \alpha, \quad \text{for all $\alpha \in [0,1]$}. \]
\end{defn}

The well-known characterization \citep[e.g.,][]{cuoso.etal.2001, destercke.dubois.2014} of the credal set of a possibility measure says that $\prob \in \cred(\Pi)$ if and only if
\[ \prob\{\pi(Y) \leq \alpha\} \leq \alpha, \quad \text{all $\alpha \in [0,1]$}, \]
i.e., if the $\prob$-probability assigned to the $\alpha$-lower level sets of $\pi$ is no more than $\alpha$ for every $\alpha \in [0,1]$.  So a maximal element of the credal set is maximal in the sense that it assigns exactly the maximal probability to these lower level sets.  Alternatively, the maximality condition implies that $\pi$ characterizes the tails of $\prob$ in the sense that 
\[ \pi(y) = \prob\{\pi(Y) \leq \pi(y)\}, \quad y \in \YY. \]

The following result establishes an apparently new connection between the fiducial distribution and the possibilistic output for the valid IM described above in the invariant statistical model setting.  Specifically, the fiducial distribution $\prior_x$ is a maximal member of the IM's credal set $\cred(\uPi_x)$.  That is, for each $x \in \XX$,
\begin{equation}
\label{eq:fid.max}
\prior_x\{\pi_x(\Theta) \leq \alpha\} = \alpha, \quad \text{for all $\alpha \in [0,1]$}. 
\end{equation}
Note that $\prior_x$ and $\pi_x$ depend on $x$ only through $t(x)=(g,u)$. Recall that, for the models under consideration here, the fiducial, generalized fiducial, and default-prior Bayes solutions all agree, so my comparison with the IM output applies to all of these.  

\begin{thm}
\label{thm:char}
Under the setup described in Sections~\ref{SS:invariant}--\ref{SS:im.group}, for each fixed $x$, the fiducial distribution $\prior_x$ is a maximal member of the IM's credal set $\cred(\uPi_x)$.
\end{thm}

\begin{proof}
The fiducial $\prior_x$-probability can be rewritten as 
\[ \prior_x\{\pi_x(\Theta) \leq \alpha\} = \prob\{ \pi_{g|u}(g \circ H^{-1}) \leq \alpha \mid U=u\}. \]
Let $z_\alpha(u)$ denote the $\alpha$-quantile of the conditional distribution of $f(H,u)$, given $U=u$; this doesn't depend on the unknown parameter.  Since $\pi_{g|u}(g \circ h^{-1}) \leq \alpha$ if and only if $f((g \circ h^{-1})^{-1} \circ g, u) \equiv f(h,u) \leq z_\alpha(u)$, it follows that 
\begin{align*}
\prior_x\{\pi_x(\Theta) \leq \alpha\} & = \prob\{ \pi_{g|u}(g \circ H^{-1}) \leq \alpha \mid U=u\} \\
& = \prob\{ f(H, u) \leq z_\alpha(u) \mid U=u\} \\
& = \alpha, 
\end{align*}
which proves the claim \eqref{eq:fid.max}.  
\end{proof}

Theorem~\ref{thm:char} shows that when the statistical model is invariant with respect to a group of transformations, the standard fiducial solution (which agrees with the default-prior Bayesian solution, among others) is the best probabilistic approximation to the IM's possibilistic output.  This eliminates any mystery surrounding the fiducial solution---it's a precise-probabilistic approximation to the valid and inherently imprecise-probabilistic IM output. That is, the fiducial argument isn't magically solving the problem of prior-free probabilistic inference within the theory of precise probability, it's just the ``best one can do'' within the theory of precise probability.  Common sense suggests that there's no free lunch and, indeed, the cost of artificially introducing precision is that one must limit the scope of questions that can be reliably answered; see Section~\ref{S:false.confidence} below.  It's true that the IM's possibilistic output lacks the richness or granularity of probabilistic output, but my claim is that, on its own, without the assistance of prior information, data can only reliably support possibilistic reasoning.  Pushing data any further than this, as the fiducial-like solutions do, creates a risk of inferences being unreliable.  


There's further insights coming from Theorem~\ref{thm:char} on the philosophical/conceptual level.  As I mentioned in Section~\ref{SS:setup}, the starting point is that there are no non-trivial, {\em a priori} probability statements that can be made about $\Theta$.  So, whether it be via the fiducial argument or something else, the data-dependent probabilities come out of thin air.  In the invariant model case, one might be able to justify the use of these probabilities based on the fact that there's a consensus---most/all of the available approaches identify the same posterior distribution.  But this doesn't help to understand what the probabilities {\em mean}.  They're clearly neither subjective nor objective/frequentist and, to my knowledge, there is no agreed-upon interpretation for these kinds of probabilities.  This isn't a surprise in light of Theorem~\ref{thm:char}.  The result says that what's driving the solution is the weaker, IM-based possibilistic uncertainty quantification. The fiducial probability is just one of many precise probabilities compatible with the IM solution's possibility measure and, therefore, can't have any genuine meaning.


\subsection{Example}
\label{SS:circle}

To illustrate the above result, and to help shed light on the practical side of this abstract formulation, consider an example representative of those encountered in directional data analysis.  The setup I'll consider here is direction measurements on the plane, which can be represented by angles relative to a reference point.  Real-world examples of this include wind direction and animal movement studies; see \citet{mardia.jupp.book} for a comprehensive treatment to the theory, methods, and applications of directional statistics.  More generally, direction measurements can be represented by points on the surface of a hypersphere.  This kind of data is common in astronomy, where the position of planets, stars, etc.~can be described by a point on the celestial sphere.  

As a very simple illustration, I'll consider the roulette wheel data from Example~1.1 in \citet{mardia.jupp.book}.  The experiment proceeds by spinning a roulette wheel and recording the angle (in radians) of the position at which the wheel stops.  Let $Y^n = (Y_1,\ldots,Y_n)$ be the observable angles for $n$ independent spins.  The actual data in this case, which consists of $n=9$ angles, is presented graphically in Figure~\ref{fig:wheel}(a).  The model under consideration here is the {\em von Mises distribution} which has a density function 
\[ p_\theta(y) = \{2\pi I_0(\kappa)\}^{-1} \exp\{\kappa \cos(y - \theta)\}, \quad y,\theta \in [0, 2\pi). \]
Here $\kappa > 0$ is a known concentration parameter, and $I_0$ denotes the Bessel function of order 0.  It's no trouble allowing $\kappa$ to be unknown as well, but I won't do so here because it takes us outside of the invariant statistical model context considered in Section~\ref{SS:invariant} above.  The parameter $\Theta$ is an unknown mean angle to be inferred from the data.  As a quick check, with $\kappa=2$ taken as fixed, the best-fit model, corresponding to the maximum likelihood estimator of $\Theta$, is shown overlaid on a histogram of the data in Figure~\ref{fig:wheel}(b).  With only $n=9$ data points, there's no reason to doubt the von Mises model.  

\begin{figure}[t]
\begin{center}
\subfigure[Observed angles]{\scalebox{0.6}{\includegraphics{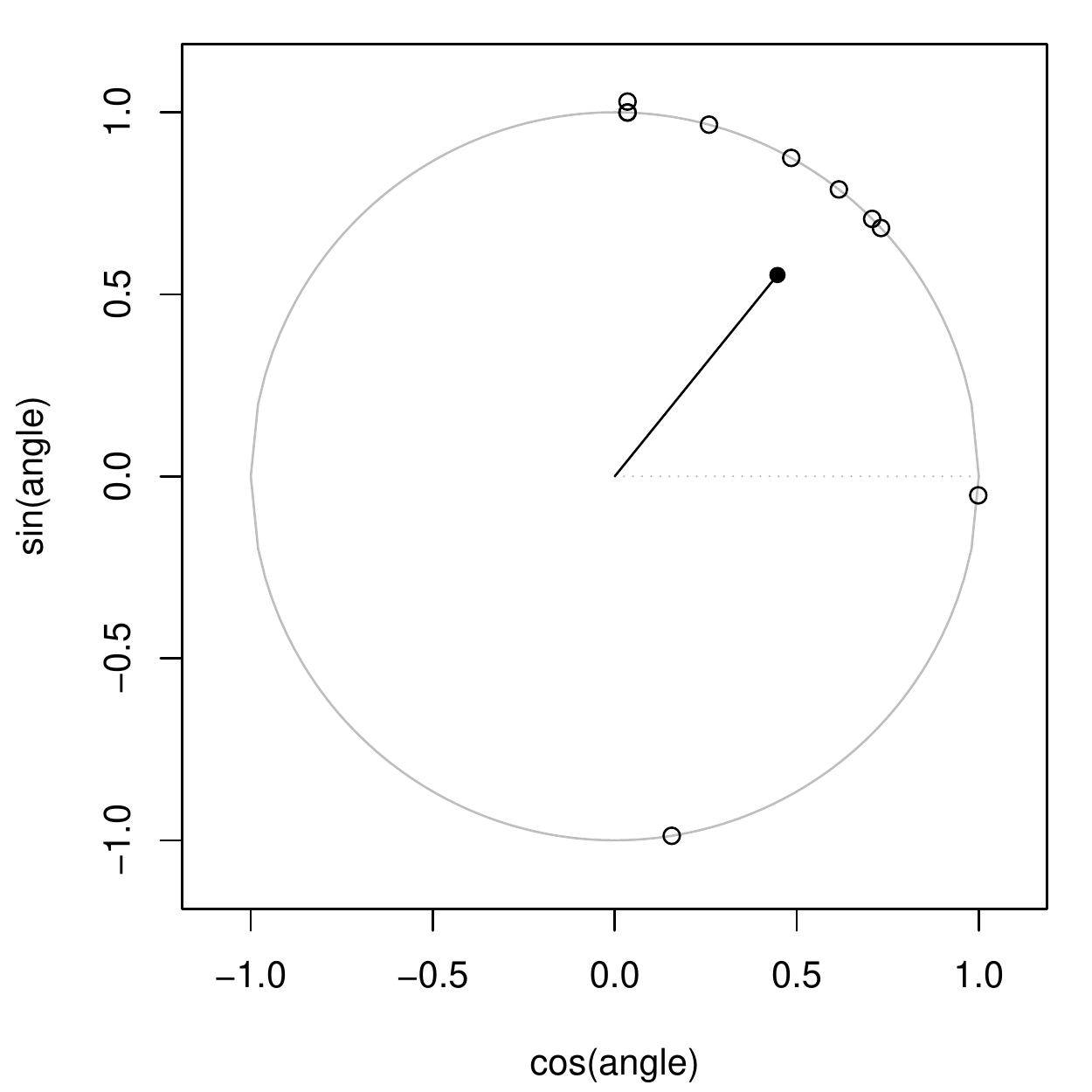}}}
\subfigure[von Mises model fit]{\scalebox{0.6}{\includegraphics{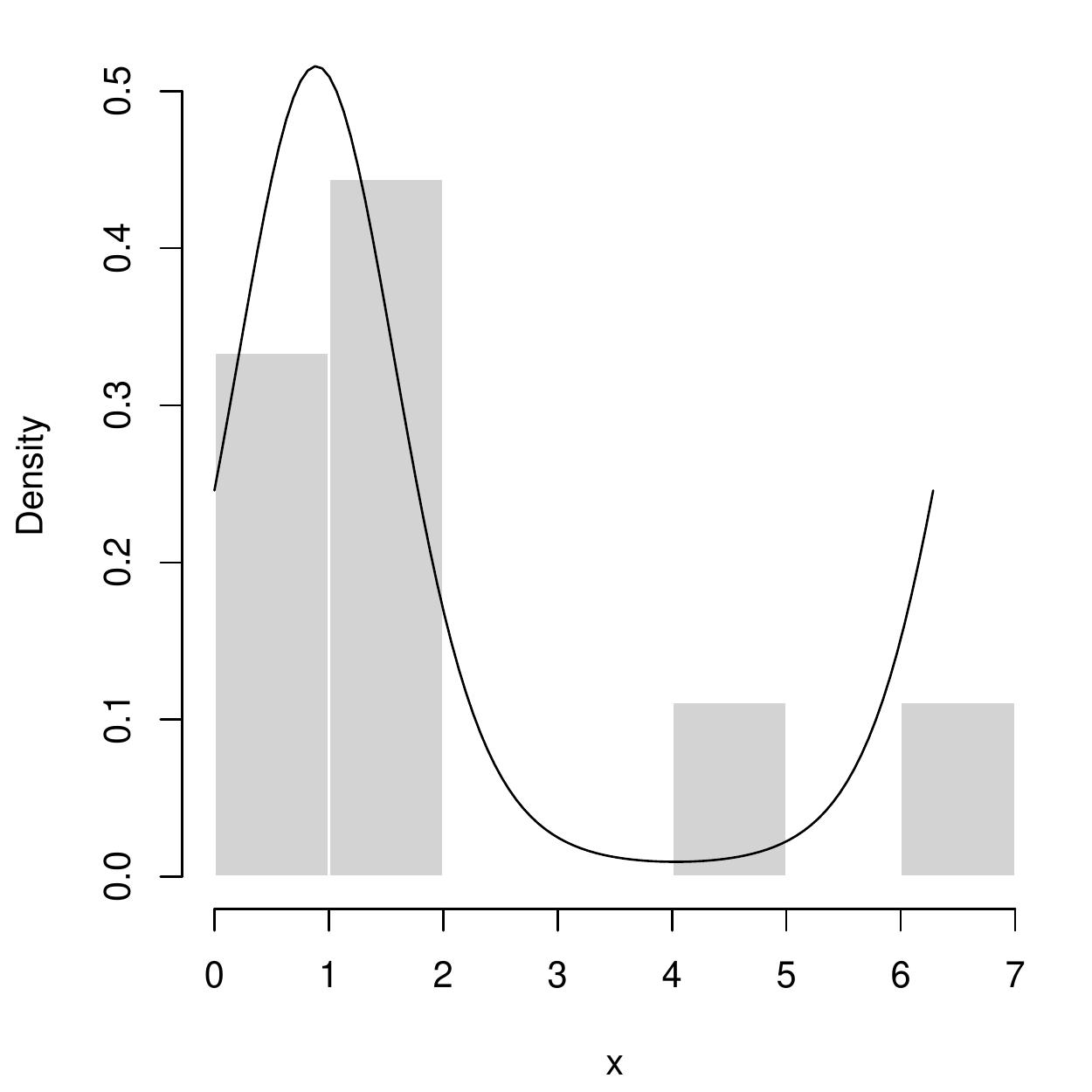}}}
\end{center}
\caption{In Panel~(a), open circles are the Cartesian coordinates corresponding to the observed roulette wheel angles; two duplicate observations near the north pole are stacked. The solid circle is the arithmetic mean of those Cartesian coordinates, and the angle the black line makes with the dotted reference line is the maximum likelihood estimator $G$ of $\Theta$; the black line's length is the sample concentration $U$.}
\label{fig:wheel}
\end{figure}

Digging into this model further, using the cosine formula 
\[ \cos(y - \theta) = \cos y \cos\theta + \sin y \sin \theta, \]
it's easy to see that the minimal sufficient statistic for this model is the pair 
\[ X = (\bar C, \bar S) = \Bigl( \frac1n \sum_{i=1}^n \cos Y_i, \, \frac1n\sum_{i=1}^n \sin Y_i \Bigr). \]
Relative to the posited von Mises model, $X$ can be treated as ``the data.'' Since $(C_i, S_i) = (\cos Y_i, \sin Y_i)$ denotes the Cartesian coordinates of the point on the unit circle corresponding to the angle $Y_i$, it's clear that $X=(\bar C, \bar S)$ is generally inside the unit circle; see Figure~\ref{fig:wheel}(a).  The more ``concentrated'' the original data points are around a particular angle, the closer $X$ will be to the boundary of the circle.  Convert this average position to polar coordinates:
\[ G = \arctan(\bar S / \bar C) \quad \text{and} \quad U = \{ \bar C^2 + \bar S^2\}^{1/2}. \]
Here $G$ is the angle the vector $(\bar C, \bar S)$ makes with the horizontal axis and $U$ is the Euclidean length of $(\bar C, \bar S)$.  It's clear that $(G,U)$ is a bijection of the minimal sufficient statistic $X$, so write $t(X) = (G,U)$.  I'm using the notation $G$ and $U$ to align with that introduced in the previous sections.  That is, there is a group of transformations $\G$ that corresponds to rotations of points in the plane; technically, $\G$ is the special orthogonal group $\text{SO}(2)$ consisting of orthogonal matrices with unit determinant.  Then $G$ can be interpreted as both an angle or as a rotation by that angle.  Similarly, the interior of the unit circle can be partitioned into distinct concentric circles, which correspond to orbits, and $U$ determines which of these orbits $X$ sits on.  

The details here align with the general setup above, so I can proceed to carry out the fiducial (or default-prior Bayes) and IM analyses.  It's a standard result \citep[e.g.,][Eq.~4.5.5]{mardia.jupp.book} in the directional statistics literature that the joint density for $X=(G,U)$, factors as 
\begin{align*}
p_\theta(g, u) & = p_\theta(g \mid u) \, p(u) \\
& = \{2\pi I_0(\kappa u)\}^{-1} \exp\{\kappa u \cos(g - \theta)\} \, p(u),
\end{align*}
where $p(u)$ denotes the marginal density for $U$, which doesn't depend on $\theta$ and isn't relevant here.  The expression for the conditional density of $G$, given $U=u$, is easily seen to be that of a von Mises distribution but with concentration parameter $\kappa u$.  The influence of conditioning on $U$ makes intuitive sense because, as indicated above, the statistic $U$ acts like a measure of how ``concentrated'' the observed angles are around a common angle.  It's the conditional density above that drives both the fiducial and IM analyses.  Indeed, the fiducial density is 
\[ q_x(\theta) \propto \exp\{\kappa u \cos(g - \theta)\}, \quad \theta \in (0,2\pi], \]
and, after simplification, the IM's possibility contour is 
\[ \pi_{g|u}(\theta) = \prob\{ \cos H \leq \cos(g - \theta) \mid U=u\},\quad \theta \in [0,2\pi), \]
where the probability is with respect to the conditional distribution of $H := G-\Theta$, given $U=u$, which is simply a von Mises distribution with mean angle 0 and concentration parameter $\kappa u$.  Plots of the fiducial density (overlaid on a sample from the fiducial distribution) and the IM's possibility contour are shown in Figure~\ref{fig:wheel.out}.  Both point to values of $\Theta$ near the maximum likelihood estimator $g=0.89$ as most plausible, as expected. 

\begin{figure}[t]
\begin{center}
\subfigure[Fiducial distribution]{\scalebox{0.6}{\includegraphics{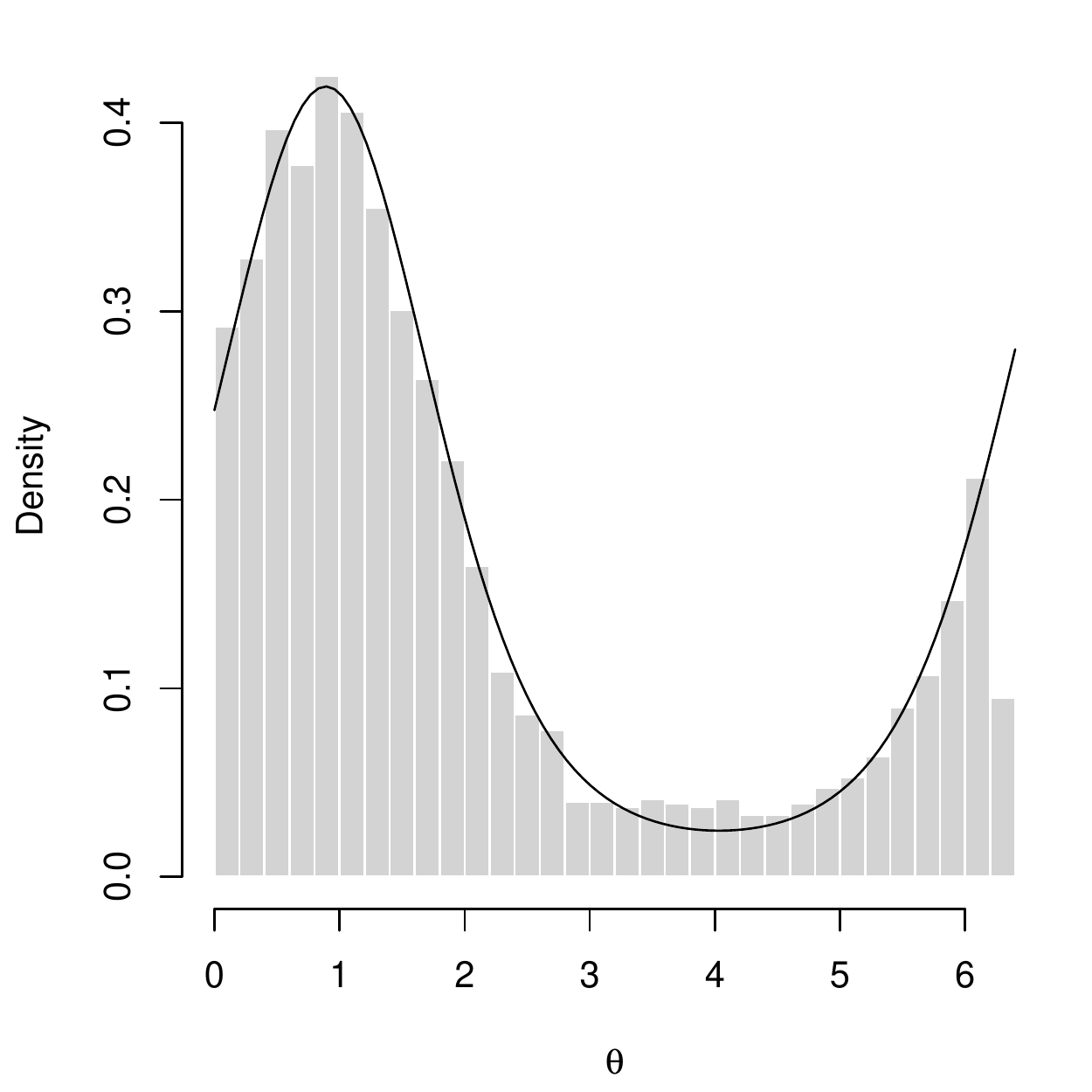}}}
\subfigure[IM possibility contour]{\scalebox{0.6}{\includegraphics{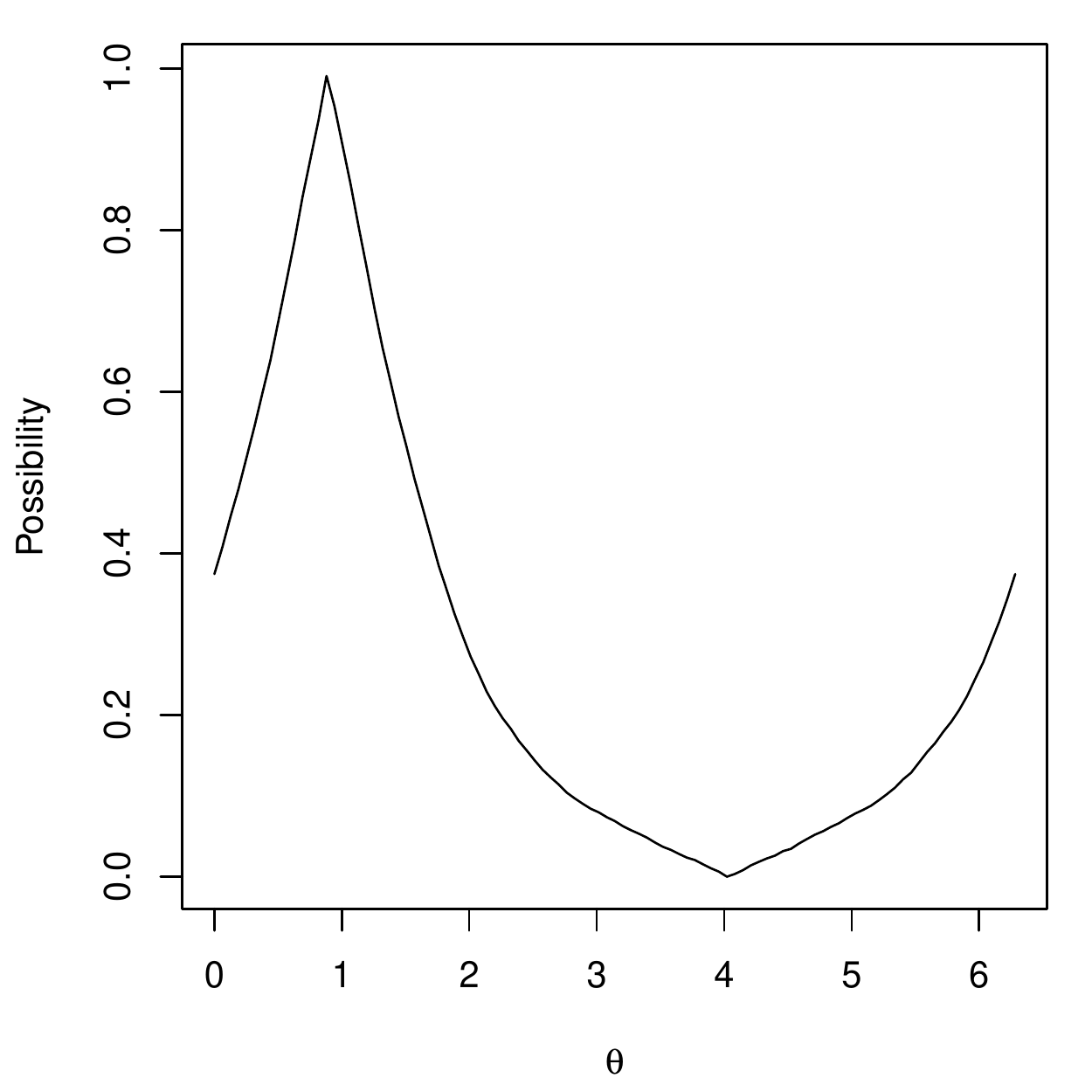}}}
\end{center}
\caption{Fiducial and IM output for inference on the angle $\Theta$ based on the roulette wheel data in Figure~\ref{fig:wheel}.}
\label{fig:wheel.out}
\end{figure}

\section{A new fiducial argument?}
\label{S:new}

Theorem~\ref{thm:char} says that, for a certain class of models, nothing changes if we {\em define} the fiducial solution as the best precise-probabilistic approximation to the IM's possibilistic output, i.e., the maximal element in the IM's credal set.  This alternative perspective on the fiducial argument is appealing for several reasons; here I'll focus on just one.  

Whether it's in the invariant statistical model setting or not, we can always {\em define} the fiducial solution to be the maximal element in the IM's credal set; I suggested a similar approach in \citet[][Sec.~3]{imdec}, and \citet[][Sec.~2]{taraldsen.2021.cd} suggested something similar.  That is, instead of formulating a way to construct a data-dependent distribution $\prior_x$ and hoping that it satisfies certain properties (e.g., a Bernstein--von Mises theorem), let's just define the fiducial distribution to be the best probabilistic approximation of the IM's possibilistic output.  This would ensure that the fiducial distribution's credible sets are genuine confidence sets, which is what ``confidence distributions'' aim to achieve \citep[e.g.,][]{thornton.xie.cd, nadarajah.etal.2015, xie.singh.2012}.  The challenge with the suggested strategy is actually finding the best probabilistic approximation.  Outside the invariant case, the currently-available fiducial solutions likely don't correspond to members of the IM's credal set.  So identifying and numerically evaluating the ``fiducial distribution'' I just defined could be a challenge.  

There are some cases in which this can be carried out.  Presently, I only have ideas for scalar-$\Theta$ cases, but I expect some generalizations are possible.  For example, one can apply the possibility-to-probability transform in \citet[][Sec.~3.2]{dubois.etal.2004} to find a ``maximal'' data-dependent probability distribution $\prior_x$, in the sense of Definition~\ref{def:maxdif} from the IM's possibility measure output $\uPi_x$, provided that the latter is unimodal.  Even in the scalar-$\Theta$ case, this can be quite complicated---I was only able to get a closed-form expression in examples that have the group invariance structure, where I already know how to get the fiducial solution.  Numerical solutions ought to be available more generally, but I haven't investigated this seriously.  


To be clear: {\em I'm not advocating for any fiducial argument}.  The validity property described in Section~\ref{SS:im} is essential to the logic of statistical inference and, as shown in \citet{balch.martin.ferson.2017}, there are no data-dependent probability distributions that can achieve it.  My goal here is simply to better understand the fiducial argument.  Indeed, it's now clear that, at least in the class of problems considered above, the fiducial solution corresponds to a maximal probabilistic approximation of the IM's possibilistic solution.  The level sets of the IM's possibility contour are confidence regions and the connection established in Theorem~\ref{thm:char} implies that these agree the fiducial confidence regions.  
The point I want to emphasize is that {\em the connection between fiducial and IMs doesn't imply that the former satisfies the same validity properties as the latter}.  The fact is, fiducial and other precise probabilistic solutions for inference on $\Theta$ are still at risk for false confidence, as I explain in Section~\ref{S:false.confidence}.  The reason is that the maximal probabilistic approximation of a marginal IM generally isn't the corresponding marginal fiducial distribution.  

For example, Figure~\ref{fig:wheel.marg} shows the marginal fiducial (and default-prior Bayes posterior) distribution and marginal IM possibility contour for $\Phi = \cos\Theta$ from the directional data example in Section~\ref{SS:circle}.  Despite the two having the same shapes in Figure~\ref{fig:wheel.out}, it's clear that (non-linear) marginalization has changed their shapes in different ways.  Here the fiducial distribution's mode is pushed fairly close to the boundary, whereas the IM's mode is at the maximum likelihood estimator $\cos 0.89 = 0.63$; since the modes differ, the fiducial distribution is no longer the best probabilistic approximation of the IM's possibility measure.  This post-marginalization discrepancy between the two methods' summaries is the result of differences between the probability and possibility calculi.  The possibility calculus is guaranteed to preserve the IM's validity property, whereas the probability calculus isn't guaranteed to preserve any such properties about the fiducial or default-prior Bayes solution. 

\begin{figure}[t]
\begin{center}
\subfigure[Marginal fiducial samples]{\scalebox{0.6}{\includegraphics{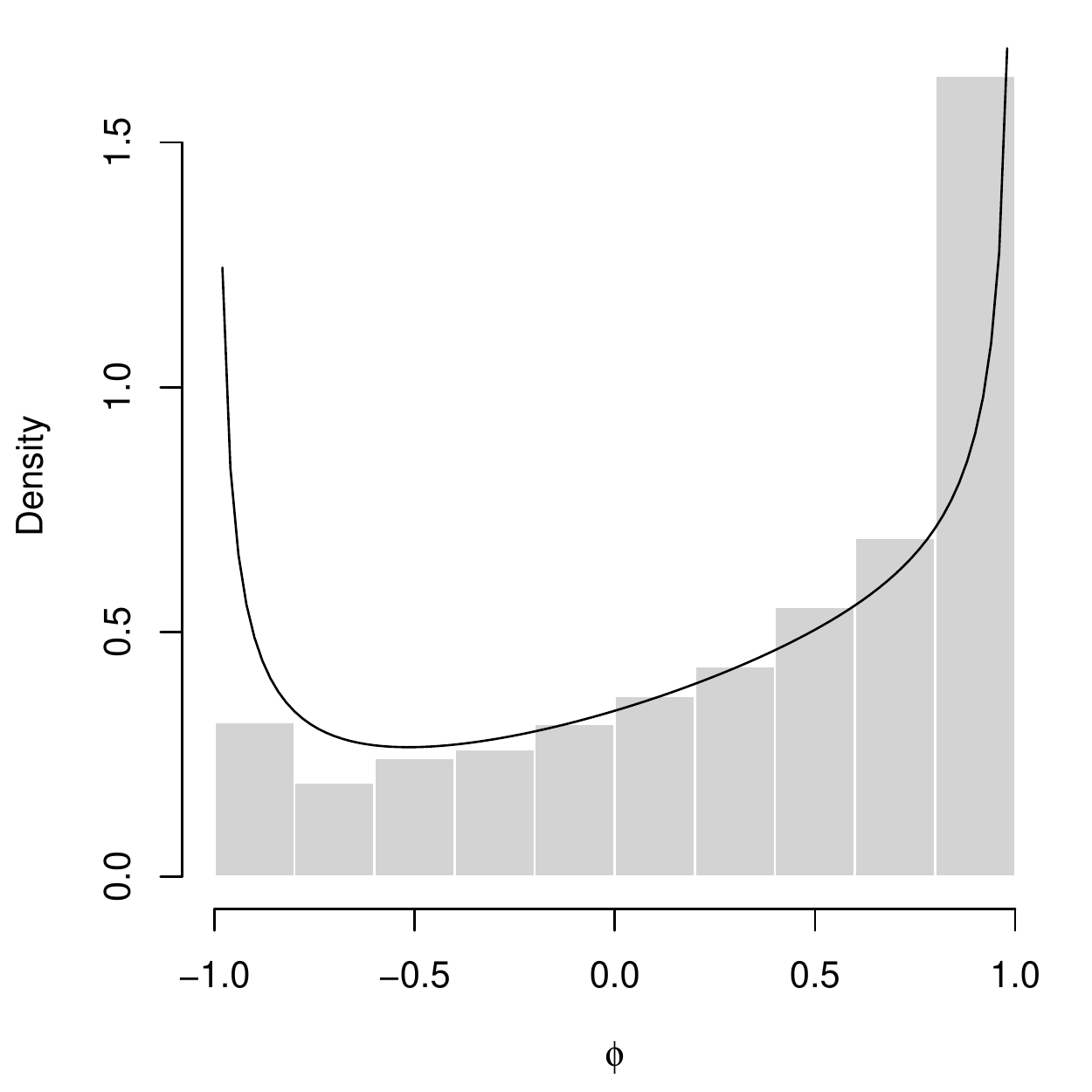}}}
\subfigure[Marginal IM possibility contour]{\scalebox{0.6}{\includegraphics{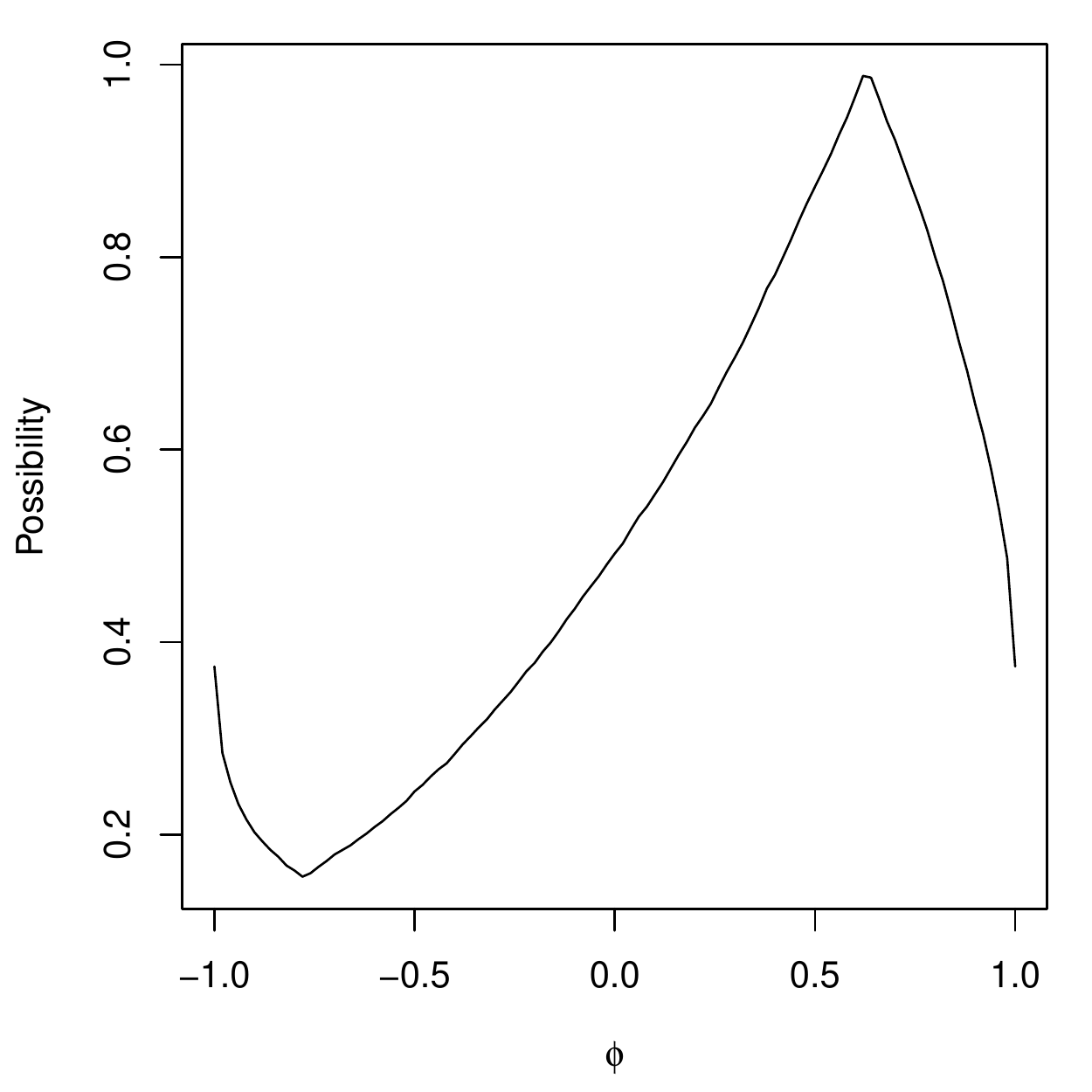}}}
\end{center}
\caption{Marginal fiducial distribution and marginal IM possibility contour for $\Phi = \cos\Theta$ derived from those for $\Theta$ as displayed in Figure~\ref{fig:wheel.out}.}
\label{fig:wheel.marg}
\end{figure}

\section{False confidence redux}
\label{S:false.confidence}

\citet{balch.martin.ferson.2017} show that data-dependent probabilities---Bayes, fiducial, etc.---used for inference suffer from what they call {\em false confidence}.  That is, if $\prior_x$ is such a distribution, then the false confidence theorem says there exists a hypothesis $A \subseteq \TT$ such that 
\begin{equation}
\label{eq:fc}
\theta \in A \quad \text{and} \quad \prob_{\theta}\{ \prior_X(A) \leq \alpha\} > \alpha, \quad \text{for some $\alpha \in [0,1]$}. 
\end{equation}
Intuitively, if $\alpha$ represents the cutoff\footnote{It's not necessary that the $\alpha$ inside and outside of \eqref{eq:fc} be the same; the inside $\alpha$ can be replaced by some $\tau(\alpha)$ for a bijection $\tau: [0,1] \to [0,1]$ with $\tau(a) \leq a$ for all $a \in [0,1]$. But the small/not-small threshold must be {\em known} to the user, otherwise it's not practically useful, so I see no compelling reason to consider anything other than $\tau(a)=a$.}  between ``small'' and ``not small,'' then \eqref{eq:fc} says that the event $\{\text{$\prior_X(A)$ is small}\}$ has not-small $\prob_{\theta}$-probability when $\theta \in A$.  Since we'd be inclined to doubt the truthfulness of $A$ when $\prior_X(A)$ is small, the property in \eqref{eq:fc} creates a risk of systematically misleading conclusions, e.g., rejecting hypotheses that are true.  Consequently, inference based on $\prior_X$ is at risk of being misleading.  

The false confidence theorem, in its current form, is only an existence result.  It could be that the only problematic $A$'s are trivial, e.g., singletons or other measure-theoretically tiny sets, not of practical interest.  Despite the motivating example in \citet{balch.martin.ferson.2017} involving non-trivial hypotheses, most statisticians have dismissed this result; apparently they don't believe non-trivial hypotheses can be afflicted.  So, the burden is on me/us to push this result to the point that the conclusion can't be denied.  

Fortunately, the structure provided by the invariant statistical model allows for a more in-depth investigation into this question.  Unfortunately, what I present below still falls short of providing a characterization of those $A$ that are afflicted with false confidence.  I can only give a (fairly general) sufficient condition for a hypothesis $A$ to be free of false confidence.  The necessary conditions that would completely settle the matter are still out of reach, but I think baby steps like this one are still valuable.  

First a bit more notation/terminology.  Let $\G$ be the group of transformations on $\XX$ described before, still with the simplifying assumption that $\TT = \G = \Gbar$, and let $\K$ be another group with binary operation $\ast$.  Suppose further that $\K$ is equipped with a total order $\lesssim$ that's bi-invariant in the sense that 
\[ a \lesssim b \implies k \ast a \lesssim k \ast b \quad \text{and} \quad a \ast k \lesssim b \ast k, \quad \text{for all $a,b,k \in \K$}. \]
Let $\psi: \G \to \K$ be a homomorphism, so that $\psi(g \circ g') = \psi(g) \ast \psi(g')$ for all $g,g' \in \G$.  Now, for a fixed $k \in \K$, define the hypothesis about $\Theta$ as 
\begin{equation}
\label{eq:special.A}
A = \{\theta: \psi(\theta) \lesssim k\}. 
\end{equation}
The claim is that hypotheses of the form \eqref{eq:special.A} are not afflicted by false confidence.  

\begin{thm}
\label{thm:fc}
For a homomorphism $\psi$ from $(\G, \circ)$ to the ordered group $(\K, \ast, \lesssim)$, and for $k \in \K$, the hypothesis $A=A_{\psi,k}$ in \eqref{eq:special.A} is {\em not} afflicted with false confidence, i.e., 
\[ \sup_{\theta \in A} \prob_{\theta}\{ \prior_X(A) \leq \alpha\} \leq \alpha, \quad \text{for all $\alpha \in [0,1]$}. \]
\end{thm}

\begin{proof}
Take $A$ as in \eqref{eq:special.A} and fix any $\theta \in A$.  Then $\psi(\theta) \lesssim k$ and, hence, $k^{-1} \ast \psi(\theta) \lesssim e$, where $e$ here denotes the identity element in $\K$.  Recall that the fiducial distribution $\Theta \sim \prior_x$ equals the conditional distribution of $g \circ H^{-1}$, given $U=u$, which depends on the values $(g,u)$ of $t(x)$ but not on $\theta$.  Since $\lesssim$ is bi-invariant and $\psi$ is a homomorphism, we find that $\psi(\Theta) \lesssim k$ is equivalent to $\psi(H) \gtrsim k^{-1} \ast \psi(g)$ and, consequently, 
\begin{align*}
\prior_x(A) \leq \alpha & \iff \prob\{\psi(H) \gtrsim k^{-1} \ast \psi(g) \mid U=y\} \leq \alpha \\
& \iff k^{-1} \ast \psi(g) \gtrsim z_\alpha(u), 
\end{align*}
where $z_\alpha(u)$ is the $(1-\alpha)$-quantile of $\psi(H)$, relative to $\lesssim$, based on the conditional distribution of $H$, given $U=u$, derived from $X \sim \prob_\theta$.  But the distribution of $G$, as a function of $X \sim \prob_{\theta}$, given that the $U$-component of $t(X)$ is $u$, is exactly the distribution of $\theta \circ H$, as a function of $H$, given $U=u$.  Therefore, 
\begin{align}
\prob_{\theta}\{ \prior_X(A) \leq \alpha \} & = \prob\{ k^{-1} \ast \psi(\theta \circ H) \gtrsim z_\alpha(u) \mid U=u\} \notag \\
& = \prob\{ k^{-1} \ast \psi(\theta) \ast \psi(H) \gtrsim z_\alpha(u) \mid U=u\} \notag \\
& \leq \prob\{ \psi(H) \gtrsim z_\alpha(u) \mid U=u \}, \label{eq:quantile}
\end{align}
where the inequality in \eqref{eq:quantile} follows from the fact that 
\[ k^{-1} \ast \psi(\theta) \ast \psi(H) \gtrsim z_\alpha(u) \impliedby \psi(H) \gtrsim z_\alpha(u), \]
which, in turn, is a consequence of the fact that 
\[ \theta \in A \iff \psi(\theta) \lesssim k \iff k^{-1} \ast \psi(\theta) \lesssim e, \]
where $e$ is the identity element in $\K$.  By definition of $z_\alpha(u)$, the probability in \eqref{eq:quantile} is no more than $\alpha$, which completes the proof. 
\end{proof}

I apologize for the level of abstraction, but this formulation helps to pinpoint the kind of structure that's incompatible with false confidence.  To make this result more tangible, and without much loss of generality, one can think of $(\K, \ast, \lesssim)$ as real numbers under addition.  Then the most natural kinds of homomorphisms in this case would be linear functions of $\theta$.  Then Theorem~\ref{thm:fc} says that hypotheses concerning linear functions of $\theta$ are free from false confidence, e.g., if $\theta$ is a $D$-vector, then the fiducial (and Bayesian) probabilities assigned to hypotheses $A=\{\theta: \theta_d \leq k\}$ for $k \in \RR$ and $d=1,\ldots,D$ are reliable for inference.  Non-linear functions, such as $\psi(\theta) = \|\theta\|$ from the motivating example in \citet{balch.martin.ferson.2017}, may not be homomorphisms and, therefore, the corresponding hypotheses are not protected from false confidence by Theorem~\ref{thm:fc} above.  Compare the difference between linear and non-linear functions of $\theta$ here, as it pertains to reliability of inference, to the discussion in \citet{fraser2011}. 




\section{Conclusion}
\label{S:discuss}

In the case of those invariant statistical models considered here, it's well-known that the fiducial and default (right Haar) prior Bayes solutions agree and, moreover, that the standard credible sets derived from these are also exact confidence regions.  Under the same setup, I showed here that the fiducial solution can also be viewed as the ``best'' probabilistic approximation to the valid IM's possibilistic solution.  This result sheds some important new light on the relationship between fiducial and IM solutions.  Indeed, if one's only concerned with, say, confidence regions for $\Theta$, then both solutions give the same results and, hence, there's no need to bother with the IM's and imprecision.  But if one is interested in other questions, e.g., inference about certain features $\psi(\Theta)$ of $\Theta$, then the connection between the fiducial and IM solutions is broken and whatever reliability the former might have for inference about $\Theta$ is lost when marginalized to $\psi(\Theta)$.  The result in Section~\ref{S:false.confidence} offers some new insights on the kinds of features $\psi$ for which the fiducial solution's reliability isn't lost in marginalization.  

Beyond the invariant statistical models, it's suggested that one could {\em define} a fiducial distribution to be the maximal probabilistic approximation to the IM solution's possibilistic output.  Just like in the invariant model case, this definition of a fiducial distribution would ensure that fiducial credible regions are exact confidence regions.  But this strong calibration comes at a cost, since it's unclear how to identify and numerically evaluate the solution to that optimization problem.  In that case, one might prefer an existing constructive solution, e.g., generalized fiducial, but this would likely not be a member of the IM's credal set and, therefore, the aforementioned confidence-connection can at best be achieved in an asymptotically approximate sense.  And just like above, the strong validity property achieved by the IM is out of reach for all fiducial solutions. 

Since the validity property is fundamental to reliable uncertainty quantification, my claim is that the IM solution is objectively better than the Bayes/fiducial solution.  So, if one opts for the latter instead of the former, then it's solely for the familiarity and/or convenience of ordinary probability theory, and this doesn't come without the cost of a potential loss of reliability.  Moreover, as the computational tools for evaluating the IM output continue to develop, this convenience gap will likewise continue to close to the point that there's no justification to sacrifice on reliability.

\section*{Acknowledgments}

This work is supported by the U.S.~National Science Foundation, grant SES--2051225.

%
%

\bibliographystyle{apalike}
\bibliography{/Users/rgmarti3/Dropbox/Research/mybib}

\end{document}